\documentclass[12pt,leqno]{amsart}%
\usepackage{amsmath}
\usepackage[
margin=1.8in,
]{geometry}
\usepackage{graphicx}
\usepackage{amsfonts}
\usepackage{amssymb}%
\usepackage[utf8]{inputenc}
\usepackage{comment}
\usepackage{hyperref}
\usepackage{mathtools}
\usepackage{ dsfont }

\newtheorem{theorem}{Theorem}[section]
\theoremstyle{plain}

\newtheorem{claim}{Claim}

\newtheorem{corollary}{Corollary}[section]

\newtheorem{proposition}{Proposition}[section]

\numberwithin{equation}{section}
\theoremstyle{definition}
\newtheorem{definition}{Definition}
\newtheorem{remark}{Remark}[section]

\def\e{\varepsilon}

\def\pd{\partial}
\def\re{\mathbb{R}}

\newcommand{\eqal}[1]{\begin{equation}\begin{aligned}#1\end{aligned}\end{equation}}

\begin{document}

\title[Convex solutions of Lagrangian mean curvature equation]{Regularity for convex viscosity solutions of Lagrangian mean curvature equation}
\author{Arunima Bhattacharya}
\address{Department of Mathematics, Phillips Hall\\
 the University of North Carolina at Chapel Hill, NC }
\email{arunimab@unc.edu}

\author{Ravi Shankar}
\address{Department of Mathematics, Fine Hall\\
Princeton University, Princeton, NJ}
\email{rs1838@princeton.edu}


\maketitle

\begin{abstract}
We show that convex viscosity solutions of the Lagrangian mean curvature equation are regular if the Lagrangian phase has H\"older continuous second derivatives.  


\end{abstract}

\section{Introduction}
We establish regularity for convex viscosity solutions of the Lagrangian mean curvature equation 
\begin{equation}
    F(D^{2}u)=\sum _{i=1}^{n}\arctan \lambda_{i}=\psi(x) \label{slag}
\end{equation}
under the assumption that $\psi:B_1(0)\to [0,n\pi/2)$ is in $C^{2,\alpha}(B_1(0))$ for some $\alpha\in(0,1)$. Here $\lambda_i$'s are the eigenvalues of the Hessian matrix $D^2u$ and then the phase $\psi$  becomes a potential for the mean curvature of the Lagrangian submanifold  $L=(x,Du(x))\subset \mathbb{C}^n$. The induced Riemannian metric $g$ can be written as 
$g=I_n+(D^2u)^2
$ or 
\[g_{ij}=\delta_{ij}+u_{ik}\delta^{kl}u_{lk}.
\]
In \cite[(2.19)]{HL}, the mean curvature vector $\vec{H}$ of this Lagrangian submanifold was shown to be
\begin{equation}
\vec{H}=J\nabla_g\psi \label{mean}
\end{equation}
 where $\nabla_g$ is the gradient operator for the metric $g$ and $J$ is the complex structure, or the $\frac{\pi}{2}$ rotation matrix in  $\mathbb{C}^n$. Note that by our assumption on $\psi$, $|H|$ is bounded. 

\smallskip
When the phase is constant, denoted by $c$, $u$ solves the special Lagrangian equation 
 \begin{equation}
\sum _{i=1}^{n}\arctan \lambda_{i}=c \label{s1}
\end{equation}or equivalently,
\[ \cos c \sum_{1\leq 2k+1\leq n} (-1)^k\sigma_{2k+1}-\sin c \sum_{0\leq 2k\leq n} (-1)^k\sigma_{2k}=0.
\]
 Equation (\ref{s1}) originates in the special Lagrangian geometry of Harvey-
Lawson \cite{HL}. The Lagrangian graph $(x,Du(x)) \subset \mathbb{C}^n$ is called special
when the argument of the complex number $(1+i\lambda_1)...(1+i\lambda_n)$
or the phase $\psi$ is constant, and it is special if and only if $(x,Du(x))$ is a
(volume minimizing) minimal surface in $(\mathbb{C}^n,dx^2+dy^2)$ \cite{HL}. 

\smallskip
 A dual form of (\ref{s1}) is the Monge-Amp\`ere equation
\begin{equation}
    \sum_{i=1}^n \ln\lambda_i=c.\label{MA}
\end{equation}
This is the potential equation for special Lagrangian submanifolds in $(\mathbb {C}^n, dxdy)$ as introduced in \cite{Hi}. The gradient graph $(x,Du(x))$ is volume maximizing in this pseudo-Euclidean space, as shown by Warren \cite{W}. In the 1980s, Mealy \cite{Me} showed that an equivalent algebraic form of the above equation is the potential equation for his volume maximizing special Lagrangian submanifolds in $(\mathbb C^n, dx^2-dy^2)$.

\smallskip
The regularity of solutions is a fundamental problem for these geometrically and analytically significant equations.  Our main results are the following:
\begin{theorem}
\label{thm}
Let $u$ be a convex viscosity solution of \eqref{slag} on $B_1(0)\subset\mathbb R^n$, where $\psi\in C^{2,\alpha}(B_1)$.  Then $u\in C^{4,\alpha}(B_1)$.
\end{theorem}

We use the above regularity to prove the following interior estimate.

\begin{theorem}
\label{thm:est}
Let $u$ be a $C^{4,\alpha}$ convex solution of \eqref{slag} on $B_1(0)\subset\mathbb R^n$, where $\psi\in C^{2,\alpha}(B_1)$.  Then $u$ satisfies the following Hessian estimate
\begin{align}
\label{hess_est}
|D^2u(0)|\le C(n,\alpha, osc_{B_{1/2}}(u), \|\psi\|_{C^{2,\alpha}(B_{1/2})}).
\end{align}
\end{theorem}

\medskip
We state the following 
application of our result. 
\vspace{0.25cm}

\noindent
\textbf{Application. }
\textit{Lagrangian mean curvature flow:} \\
We specialize to $\psi(x)=\kappa\cdot x+c$.  In this case, the gradient graph $(x,Du(x))$ corresponds to a translating soliton for Lagrangian mean curvature flow. Indeed, if $u(x,t)$ solves the potential equation
$$
u_t=\sum_{i=1}^n \arctan\lambda_i(x,t),
$$
then the gradient graph $X(t)=(x,Du(x,t))$ solves the Lagrangian mean curvature equation
$$
(X_t)^\bot=H,
$$
where $H=J\nabla_g\psi$ is the mean curvature, and $\bot$ is the projection onto the normal space of $\{(x,Du(x,t))\}_x\subset\mathbb R^n\times\mathbb R^n$; see \cite[pg 203]{CCH}. Now let $u(x,0)=u(x)$ solve \eqref{slag} for $\psi(x)=\kappa\cdot x+c$. Then $u(x,t)=u(x)+t(\kappa\cdot x+c)$, so $(x,Du(x,t))=(x,Du(x))+t(0,\kappa)$ is a translating soliton with ``constant" mean curvature $\kappa$. We conclude that translating Lagrangian solitons with convex potentials and ``vertical" speeds $(0,\kappa)$ are smooth. Local a priori estimates for convex solutions to Lagrangian mean curvature flow were found in \cite{NY}.

\smallskip
The convexity of the special Lagrangian equation plays a dominant role 
in its regularity theory.  
The arctangent operator is concave if $u$ is convex, or if the Hessian has a lower bound $\lambda\ge 0$.  Bao and Chen \cite{BCconvex} showed regularity for convex $W^{2,p}$ strong solutions.  For smooth convex solutions of (\ref{s1}), Hessian estimates have been obtained by Chen-Warren-Yuan \cite{WYJ}. Recently in \cite{CSY}, Chen-Shankar-Yuan confirmed that convex viscosity solutions of \eqref{s1} are smooth. The semiconvex singular solutions constructed by Wang-Yuan \cite{WdY} show that the convexity condition is necessary.  Similarly, if in \eqref{s1} we have critical phase $|c|=(n-2)\pi/2$ or supercritical phase $|c|>(n-2)\pi/2$, then $F(D^2u)$ has convex level sets, but it was shown by Yuan \cite{YY0} that this fails for subcritical phases $|c|<(n-2)\pi/2$.  For critical and supercritical phases, Hessian estimates for \eqref{s1} have been obtained by Warren-Yuan \cite{WY9,WY} and Wang-Yuan \cite{WaY}, see also \cite{L}. For subcritical phases, $C^{1,\alpha}$ solutions of \eqref{s1} were constructed by Nadirashvili-Vl\u{a}du\c{t} \cite{NV} and Wang-Yuan \cite{WdY}.  Hessian lower bounds also play a role in the dual equation \eqref{MA}.  The Monge-Amp\'ere type equation $\det D^2u=h(x)$ has Pogorelov-type singular $C^{1-2/n}$ convex viscosity solutions whose graphs contain a line.  However, under the necessary \textit{strict} convexity assumption $\lambda>0$, interior regularity was obtained by Pogorelov \cite{P2} for a smooth enough right hand side $h(x)$, by Urbas \cite{U} if $h(x)$ is Lipschitz and $u(x)\in C^{(1-2/n)^+}$, and by Caffarelli \cite{Caf} if $h(x)$ is merely H\"older.  

\smallskip
The convexity 
of $u$ creates its own challenges in proving regularity, since it is unstable under smooth approximations of the boundary value.  To use the method of a priori estimates, one would solve the Dirichlet problem for a modified concave equation $\Tilde{F}(D^2u)=\psi(x)$, where $\tilde f(\lambda)=\arctan\lambda$ for $\lambda \geq 0$ and $=\lambda$ for $\lambda<0$.  If we approximate the boundary data and find the smooth solution of this problem, it may no longer be convex, and therefore lacks a connection to the original problem.  Although a priori estimates are available for \eqref{slag} and \eqref{s1}, estimates for $\tilde F$ are unknown.  

\smallskip
In showing regularity for convex solutions of \eqref{s1}, the authors of \cite{CSY} managed to avoid a priori estimates altogether.  The basic idea was to change variables using the Lewy-Yuan rotation of the gradient graph, $\bar x+iD\bar u(\bar x)=e^{-i\pi/4}(x+iDu(x))$, such that the Hessian bounds decreased from $0\le \lambda\le+\infty$ to $-1\le\bar\lambda\le 1$. Since minimal graphical tangent cones with such bounds are flat, they were able to deduce regularity in new coordinates using ideas from \cite{Y}. One problem, however, is that defining such a rotation is unclear if convex $u$ is not $C^1$.  To adapt to the lower regularity setting, it was shown in \cite{CSY} that in the smooth case, $\bar u(\bar x)$ can be constructed using the Legendre transform, which is still well defined in the Lipschitz case.

\smallskip
Extending the constant phase \eqref{s1} results to variable phases \eqref{slag} is subtle, and not always possible.  Interior regularity for solutions of (\ref{slag}) with a $C^{1,1}$ supercritical and critical phase were recently obtained in \cite{AB,AB2d,BMS}, but $C^{1,\alpha}$ singular solutions are known if $\psi$ has 
H\"older regularity \cite[Remark 1.3]{AB}.  In fact, these singular solutions are convex, so Theorem \ref{thm} is not valid unless $\psi(x)$ is sufficiently smooth.  Moreover, in our companion paper \cite{BS2}, we exhibit convex $C^{1,\alpha}$ singular solutions which solve an equation with $\psi=f(x,Du)$ depending also on the gradient.  In such cases, $f$ is smooth in both arguments.  Convex singular solutions do not exist for analogous uniformly elliptic PDEs $F(D^2u)=f(x,Du)$.  The non-uniform ellipticity of the arctangent operator in \eqref{s1} makes the PDE highly sensitive to the structure of the variable phase.  Nevertheless, the results of this paper show that \textit{no conditions} on $\psi(x)$ are needed for regularity of convex viscosity solutions, apart from some smoothness such as $C^{2,\alpha}$.


\smallskip
In proving Theorem \ref{thm}, the treatment of the variable phase is delicate, and the technique does not always work.  The key step to deduce $C^{1,1}$ regularity, or $\lambda_{max}(D^2u)<\infty$, is to show that in rotated coordinates, $\bar \lambda_{max}$ never saturates its upper bound of $1$.  For constant phases, this is done using the strong maximum principle for $\Delta_{\bar g}$-subharmonic $\bar\lambda_{max}$.  For variable phases, subharmonicity only holds up to extra terms, such as $D^2\bar\psi(e_{max},e_{max})$.  If $\psi$ is convex, then subharmonicity is restored, but Theorem \ref{thm} does not assume any conditions on $\psi$ other than $C^{2,\alpha}$ smoothness.  To handle these terms, we carefully account for the coordinate change, and relate $D^2\bar\psi$ to $D^2\psi$, for example.  It turns out that the resulting expression vanishes when $\bar\lambda_{max}$ saturates its upper bound, restoring the maximum principle.  To understand how this could happen, observe that for a convex singular solution, the forward map $\bar x=x\cos\alpha+Du(x)\sin\alpha$ has large Jacobian matrix in the $\lambda_{max}$ directions.  It follows that the inverse map has small Jacobian in those directions, so the rotated phase $\bar\psi(\bar x)=\psi\circ x(\bar x)$ will inherit this flatness.

\smallskip
This paper is divided into the following sections.  In section \ref{sec:rot}, we formulate the Lewy-Yuan rotation for the Lipschitz potential $u(x)$. In section \ref{sec:reg_baru}, we establish regularity of the rotated potential $\bar u(\bar x)$. In section \ref{sec:reg}, we deduce regularity of the original potential $u(x)$, thereby proving Theorem \ref{thm}. In section \ref{sec:comp}, we prove the Hessian estimate \eqref{hess_est}.
 
\medskip
\textbf{Acknowledgments.} The authors are grateful to Yu Yuan for his guidance, support, and helpful discussions. The authors thank D. H. Phong for helpful comments and suggestions. AB acknowledges the support of the AMS-Simons Travel Grant. RS was partially supported by the NSF Graduate Research Fellowship Program under grant No. DGE-1762114. The authors thank the anonymous referee for the referee's thorough feedback.

\section{Rotation for Lipschitz potential}
\label{sec:rot}
We formulate Lewy-Yuan rotation for the convex function $u$ solving (\ref{slag}), using the idea introduced in \cite{CSY}; we refer to sections 2.1 and 2.2 in that paper for various details in this section. If $u(x)$ is smooth on $\Omega$, then the gradient graph $z=(x,Du(x))$ is a Lagrangian submanifold of $\mathbb C^n$, and if $u$ is convex, then the downward rotation $\bar z=e^{-i\pi/4}z$ yields another Lagrangian submanifold $(\bar x,D\bar u(\bar x))$ with smooth potential $\bar u(\bar x)$ on $\bar x(\Omega)$.  Because the canonical angles decrease by an angle of $\pi/4$
$$
\arctan\bar\lambda_i(D^2\bar u)=\arctan\lambda_i(D^2u)-\pi/4,
$$
it follows that the rotation sends solutions of \eqref{slag} to solutions $\bar u(\bar x)$ of another Lagrangian mean curvature equation, which now is uniformly elliptic by the Hessian bounds
\begin{align}
\label{bounds}
-I_n\le D^2\bar u\le I_n.
\end{align}
This makes regularity theory tractable in new coordinates using the ideas developed in \cite{Y}.

\smallskip
If we do not assume that $u$ is in $C^1$, then the subdifferential $(x,\pd u(x))$, i.e., the slopes of tangent planes touching $u$ from below at $x$,
is not a graph over $\Omega$ (for example, if $u(x)=|x|$, then $\pd u(0)=\bar B_1(0)$).
Nevertheless, if we rotate the subdifferential $(x,\pd u(x))$ downwards by $\pi/4$, then we still obtain a gradient graph $(\bar x,D\bar u(\bar x))$, where $\bar u\in C^{1,1}$ and is given explicitly by
\begin{equation}
\label{r}
    s\bar{u}(\bar{x})-\frac{c}{2}|\bar x|^2=-\left[su(x)+\frac{c}{2}|x|^2\right]^*(\bar x),\qquad \bar x\in \bar x(\Omega),
\end{equation}
where $c=\cos\pi/4$, $s=\sin\pi/4$, and
$$
f^*(\bar x)=\sup_{x}[x\cdot\bar x-f(x)],\qquad \bar x\in\pd f(\Omega)
$$
is the Legendre transform of convex $f(x)$; see \cite[Proposition 2.1]{CSY}.  The image domain $\bar\Omega=\bar x(\Omega)=\pd \tilde u(\Omega)$ is open and connected by \cite[Lemma 2.1]{CSY}, 
which comes from the fact that
\begin{equation}
\tilde u(x):=su+\frac{c}{2}|x|^2 \label{u1}
\end{equation}
is strictly convex.  

\medskip
The transform \eqref{r} is order preserving, $f\le g\to \bar f\le \bar g$, hence, it preserves uniform limits.  It follows that $\bar u(\bar x)$ satisfies the same Hessian bounds \eqref{bounds} as in the smooth case.  We could then use interpolation to see that $D\bar u_n\to D\bar u$ locally uniformly provided $u_n\to u$ uniformly, but this also follows from strict convexity, as in \cite[Proposition 2.3]{CSY}, i.e. if $\bar x\in \pd \tilde u(a)\cap\pd \tilde v(b)$ for $u,v$ convex, then
\begin{align}
\label{xsmall}
|b-a|^2\le 2\sqrt 2(|\tilde u-\tilde v|(a)+|\tilde u-\tilde v|(b)).
\end{align}
The smallness of $|D\bar u-D\bar v|$ would then follow from reverse rotation $a=-s\bar x+cD\bar u$.



\smallskip
Following \cite[Propositions 2.2,2.3]{CSY}, we show that viscosity solutions are preserved under this rotation.

\begin{proposition}
Let $u(x)$ be a convex viscosity solution of (\ref{slag}) in $B_{1.2}(0)$. Then the $\pi/4$-rotation $\bar u$ in (\ref{r}) is a corresponding viscosity solution of 
\begin{equation}
    \bar F(D^2\bar{u})=\sum_{i=1}^n \arctan \bar{\lambda}_i(D^2\bar{u})=\bar\psi(\bar x,D\bar u)=\psi(x)-n\pi/4 \label{rs}
\end{equation}
in open $\bar{\Omega}=\partial \Tilde{u}(B_1(0))$, where $\tilde{u}$ is as defined in (\ref{u1}).
\end{proposition}

\begin{proof} We first prove the following claim. 

\begin{claim} \label{1}
We show that $\bar u$ is a supersolution of (\ref{rs}) in $\bar{\Omega}$.
\end{claim}
Let $\bar{P}$ be a quadratic polynomial touching $\bar{u}$ from below locally
somewhere on the open set $\bar{\Omega},$ say at the origin
$\bar{0}\in\partial\tilde{u}(  0).$ Since at $\bar{0}$, $D^2\bar{P}\leq D^2\bar{u}\leq I_n$, we assume $D^{2}\bar{P}<I_n$, by subtracting $\varepsilon|\bar x|^2$ from $\bar{P}$, and then letting $\varepsilon\rightarrow 0$. This
guarantees the existence of its pre-rotated quadratic polynomial $P,$ and also confirms that
$\bar{P}$ touches $\bar{u}$ from below in an open neighborhood of the closed set $\partial\tilde{u}(  0).$  We still have $D\bar P(\bar 0)=D\bar u(\bar 0)$.  Using the order preservation
of $\alpha$-rotation that is also valid for reverse rotation, and continuity of the gradient mapping $\partial\tilde{u}$ in \cite[Corollary 24.5.1]{R}, we see that the pre-rotated quadratic polynomial $P$ touches $u$ at $x=0$ from below in an open neighborhood of $0\in B_{1}(  0).$ Using the fact that $u$ is a supersolution of (\ref{slag}) we get 
\[\sum_{i=1}^n\arctan\lambda_i(D^2P)\leq \psi(0), 
\]which in turn implies 
\[\bar F(\bar D^2\bar P)=\sum_{i=1}^n\arctan\bar \lambda_i(D^2\bar P)\leq \psi(0)-n\pi/4=\bar\psi(\bar 0,D\bar P(\bar 0)).
\]
Therefore, the claim holds good. 
\begin{claim}\label{2}
We show that $\bar u$ is a subsolution of (\ref{rs}) in $\bar{\Omega}$.
\end{claim}

The first part of this proof is the same as \cite[Prop 2.3, Step1]{CSY} where a convolution $u_\varepsilon$ of $u$ is considered and the smooth $\alpha$-rotation $\bar{u}_\varepsilon$ is shown to be well defined on $\overline{\Omega}=\partial \tilde{u}(B_1(0))\subset D\tilde{u}_\varepsilon (B_{1.1}(0))$ for $\varepsilon>0$ small enough.\\
Next, observe that (\ref{slag}) is concave when $u$ is a convex function. Applying \cite[Theorem 5.8]{CC}, we see that the solid convex average $u_\varepsilon=u*\rho_\varepsilon$ is a subsolution of 
\[F(D^2u_\varepsilon)\ge\psi_\varepsilon(x)=\psi*\rho_\varepsilon(x)
\] for $\varepsilon>0$ small enough. Combining \cite[Prop 2.1]{CSY} with the first part of this proof, we see that the smooth $\alpha$-rotation $\bar{u}_\varepsilon$ is a subsolution of 
\begin{equation*}
    \bar{F}(D^2\bar{u}_\varepsilon)=\sum_{i=1}^n \arctan \bar{\lambda}_i(D^2\bar{u}_\varepsilon)\ge \bar\psi_\varepsilon(\bar{x}_{\varepsilon},D\bar u_{\varepsilon})=\psi_\varepsilon(x_\e)-n\alpha,
\end{equation*}
in $\overline{\Omega}$, where $x_\e=c\bar x-s D\bar u_\e(\bar x)$.  By \eqref{xsmall}, $x_\e\to x=c\bar x-s\,D\bar u(\bar x)$ uniformly on $\bar\Omega$, so $\psi_\e(x_\e(\bar x))\to\psi(x(\bar x))$ uniformly on $\bar\Omega$.   Since $\bar u_\e\to \bar u$ uniformly on $\bar\Omega$, and they are viscosity subsolutions of locally uniformly convergent equations $\bar F(D^2\bar u_\e)=f_\e(\bar x)$, it follows from \cite[Proposition 2.9]{CC} that $\bar u$ is a subsolution of the limiting equation \eqref{rs}.  


\smallskip
Therefore, from claims \ref{1} and \ref{2}, we see that $\bar u$ is a viscosity solution of (\ref{rs}) in $\bar{\Omega}$.
\end{proof}

In the remainder, we abbreviate the rotated phase by
\eqal{
\label{barpsi}
\bar\psi(\bar x)&:=\bar\psi(\bar x,D\bar u(\bar x))=\psi(x(\bar x,D\bar u(\bar x)))-n\pi/4\\
&=\psi\left(\frac{\bar x-D\bar u(\bar x))}{\sqrt 2}\right)-n\pi/4.
}
In particular, since $\bar u\in C^{1,1}$, we see that $\bar\psi$ is Lipschitz, so far.

\section{Regularity of the rotated potential}
\label{sec:reg_baru}
We first show that $\bar u\in VMO(B_{1/2})$ followed by $\bar u\in C^{2,\alpha}(B_{1/2})$.
We recall the notion of VMO functions. 
\begin{definition}[Vanishing mean oscillation]
Let $\Omega\subset\mathbb{R}^n$. A locally integrable function $v$ is in $VMO(\Omega)$ with modulus $\omega(r,\Omega)$ if
\[\omega(r,\Omega)=\sup_{x_0\in \Omega,0<r\leq R}\frac{1}{|{B_r(x_0)\cap \Omega}|} \int_{B_r(x_0)\cap \Omega} |v(x)-v_{x_0,r}|\rightarrow 0, \text{ as $r\rightarrow 0$}
\]
where $v_{x_0,r}$ is the average of $v$ over $B_r(x_0)\cap \Omega.$
\end{definition}

\begin{proposition}[VMO Estimates]
\label{prop:vmo}
Let $\bar u$ be a $C^{1,1}$ viscosity solution of (\ref{rs}) in $B_{1}(0)\subset \mathbb R^n$, where $|D^2\bar u|\leq 1$ and $ \bar\psi(\bar x)$ is 
continuous. Then $D^2 \bar u \in VMO(B_{1/2})$ and the $VMO$ modulus of $\bar u$, denoted by $\omega(r)\rightarrow 0$ as $r\rightarrow 0$.
\end{proposition}
\begin{proof}
Suppose the contrary is true. Then we can find $\varepsilon>0$ and sequences $\{\bar x_k\rightarrow\bar x_\infty\}\subset B_{1/2}$, $\{r_k \rightarrow 0\}$, and a family of $C^{1,1}$ viscosity solutions $\{\bar u_k\}$ of (\ref{rs}) with $|D^2\bar u_k|\leq 1$, such that 
\[
\frac{1}{|B_{r_k}|}\int_{B_{r_k}}|D^2\bar u_k-(D^2\bar u_k)_{\bar x_k,r_k}|\geq \varepsilon.
\] 
Next, we blow up $\{\bar u_k\}$. For $|\bar y|\leq \frac{1}{r_k}$, we set
\[\bar v_k(\bar y)=\frac{\bar u_k(\bar x_k+r_k\bar y)-\nabla \bar u_k(\bar x_k). r_k \bar y-\bar u_k(\bar x_k)}{r_k^2}.
\]
Here, $\bar v_k$ is a viscosity solution of
\[
\sum_{i=1}^n\arctan\bar\lambda_i(D^2\bar v_k(\bar y))=\bar\psi(\bar x_k+r_k\bar y).
\]
By continuity of $\bar\psi$, the right hand side converges uniformly to $\bar\psi(\bar x_\infty)$, while the left hand side can be extended outside of $|D^2\bar v|\le 1$ to a uniformly elliptic operator.  Meanwhile, for any fixed $s>0$, we use $\|D^2\bar v_k\|_{L^\infty(B_{r_k})}\le 1$ and $\bar v_k(0)=D\bar v_k(0)=0$ to find a $C^{1,\alpha}(B_s)$ convergent subsequence.  By the diagonalization method, we find a subsequence, also denoted $\bar v_k$, which converges locally uniformly in $C^{1,\alpha}$ on $\re^n$ to $\bar v$ as $k\to\infty$.  Viscosity solutions are closed under $C^0$ uniform limits and locally uniformly convergent, uniformly elliptic sequences of PDEs \cite[Proposition 2.9]{CC}, so on any fixed ball, we find that $\bar v$ is a viscosity solution of the special Lagrangian equation \eqref{s1}
$$
\sum_{i=1}^n\arctan\bar\lambda_i(D^2\bar v(y))=\bar\psi(\bar x_\infty).
$$
We also have $W^{2,p}_{loc}$ convergence.  Applying the $W^{2,\delta}$ estimate from \cite[Prop 7.4]{CC}, we get \[ ||D^2\bar v_k-D^2\bar v||_{L^{\delta}(B_{s/2})}\leq C(s)||\bar v_k-\bar v||_{L^{\infty}(B_s)}\]where the RHS goes to $0$ as $k\rightarrow \infty$. Using the fact that $|D^2\bar v_k|, |D^2\bar v|\leq 1$, we see that for any $p>0$,  $||D^2\bar v_k-D^2\bar v||_{L^{p}(B_{s/2})}$ approaches $0$ as $k\rightarrow\infty$. 

\smallskip
Since $\bar v$ solves the special Lagrangian equation with Hessian bounds $-I_n\le D^2\bar v\le I_n$ on $\re^n$, the work of Yuan from the 2000's now applies \cite{Y}.  We refer verbatim to \cite[pg. 9, Step 1]{CSY} which shows that $\bar v$ is smooth under such conditions.  We next apply \cite[pg. 122, Step B]{YY}, which shows that a smooth entire solution of the constant phase equation with bounds \eqref{bounds} is quadratic.  We briefly summarize the idea here: Using the calibration argument, the $C^{0,1}$ gradient graph $(\bar x,D\bar v(\bar x))$ is volume minimizing.  The monotonicity formula \cite[Pg 84]{sim} combined with the flatness \cite[Prop 2.2]{YY} of graphical tangent cones with Hessian bounds $-I_n\le D^2\bar v\le I_n$, show that $\bar v$ is smooth, by applying the VMO and then $C^{2,\alpha}$ arguments from \cite{Y}.  On the other hand, the Hessian bounds also rule out tangent cones at infinity, giving the Bernstein theorem of \cite{YY}.

\smallskip
Continuing with our proof, we next use the $W^{2,p}_{loc}$ convergence and $D^2\bar v=const.$ to obtain:
\begin{align*}
0=\frac{1}{|B_1|}\int_{B_1}|D^2\bar v-(D^2\bar v)_{0,1}|=\lim_{k\rightarrow \infty}\frac{1}{|B_1|}\int_{B_1}|D^2\bar v_k-(D^2\bar v_k)_{0,1}|\\=\lim_{k\rightarrow \infty}\frac{1}{|B_{r_k}|}\int_{B_{r_k}}|D^2\bar u_k-(D^2\bar u_k)_{\bar x_k,r_k}|\ge \varepsilon,
\end{align*}
which is a contradiction. 
\end{proof}

By translation invariance of the VMO seminorms, $D^2\bar u(\bar x+h)$ will also be in $VMO(B_{1/2})$ if $h$ is small.  This means we can conclude the following but for the sake of completeness, we provide a proof similar to Proposition \ref{prop:vmo}.
\begin{corollary}
\label{cor:vmo}
Let $h\in \mathbb R^n$ be sufficiently small, and $\bar u$ and $\bar\psi$ be as defined in Proposition \ref{prop:vmo}.  Then a continuous function $K(.,.)$ of $D^2\bar u(\bar x)$ and $D^2\bar u(\bar x+h)$ is in $VMO(B_{1/2})$. 
\end{corollary}

\begin{proof}
Letting $\bar u^h_k(\bar x)=\bar u_k(\bar x+ h)$, we repeat the proof of Proposition \ref{prop:vmo}, assuming instead that
$$
\frac{1}{|B_{r_k}|}\int_{B_{r_k}}|K(D^2\bar u_k,D^2\bar u^{h_k}_k)-(K(D^2\bar u_k,D^2\bar u^{h_k}_k))_{\bar x_k,r_k}|\ge\e.
$$
Rescaling as before via $\bar v_k$ and $\bar v_k^{h_k}$, we take subsequences and send $k\to\infty$ to obtain limits $\bar v$ and $\bar v^h$.  The convergence is in $W^{2,p}_{loc}(\mathbb R^n)$, so after a subsequence, we can assume $D^2\bar v_k$ and $D^2\bar v^{h_k}_k$ converge almost everywhere.  By the dominated convergence theorem and $|D^2\bar v_k|\le 1$, it follows that $K(D^2\bar v_k,D^2\bar v_k^{h_k})$ converges in $L^1_{loc}(\mathbb R^n)$.  In fact, $\bar v$ and $\bar v^h$ are quadratic polynomials, so $K(D^2\bar v,D^2\bar v^h)$ is a constant. Therefore we have
\begin{align*}
0=\frac{1}{|B_1|}\int_{B_1}|K-(K)_{0,1}|&=\lim_{k\to\infty}\frac{1}{|B_{1}|}\int_{B_{1}}|K(D^2\bar v_k,D^2\bar v_k^{h_k})-(K(D^2\bar v_k,D^2\bar v_k^{h_k}))_{0,1}|\\
&=\lim_{k\to\infty}\frac{1}{|B_{r_k}|}\int_{B_{r_k}}|K(D^2\bar u_k,D^2\bar u^{h_k}_k)-(K(D^2\bar u_k,D^2\bar u_k^{h_k}))_{\bar x_k,r_k}|\\
&\ge\e,
\end{align*}
which is a contradiction.
\end{proof}

Note that the $VMO(B_{1/2})$ seminorms of $K(D^2\bar u,D^2\bar u^h)$ are independent of $h$ if $h$ is small.  The point of $K$ is we need to take a difference quotient of \eqref{rs}.

\begin{proposition}[$C^{2,\alpha}$ Estimates]
\label{prop:C2a}
Let $\bar u$ be a $C^{1,1}$ viscosity solution of (\ref{rs}) in $B_{1}(0)\subset \mathbb R^n$, where $|D^2\bar u|\leq 1$ and $ \bar\psi$ is Lipschitz continuous. Then $\bar u \in C^{2,\alpha}(B_{1/2})$ for all $\alpha\in(0,1)$.
\end{proposition}

\begin{proof}
Letting $v^h=[\bar u(\bar x+h)-\bar u(\bar x)]/|h|$, we obtain the linearized equation
\begin{align}
\label{slag:lin}
F^h_{ij}v^h_{ij}=\bar\psi^h(\bar x),
\end{align}
where
$$
F^h=\int_0^1\Big(I+((1-t)D^2\bar u(\bar x)+tD^2\bar u(\bar x+h))^2\Big)^{-1}dt,\qquad \bar\psi^h(\bar x)=\frac{\bar\psi(\bar x+h)-\bar\psi(\bar x)}{|h|}
$$ and we use the Einstein summation notation.
For small $h$, equation \eqref{slag:lin} holds 
pointwise (in fact, $\bar u(\bar x)$ is twice differentiable everywhere, which is easily shown using the rescaling procedure in Proposition \ref{prop:vmo}) 
on $B_{1/2}$, and can be thought of as a linear equation for $v^h$ in nondivergence form, with, by Corollary \ref{cor:vmo}, $VMO_{loc}$ coefficients $F^h_{ij}$ and $L^\infty_{loc}$ right hand side $\bar \psi^h$, each of whose seminorms are independent of $h$.  Recalling the interior $W^{2,p}$ estimates due to \cite{CFL}, see also \cite[Theorem 2.1]{V}, (although these estimates assume solutions are in $W^{2,p}_0(\Omega)$, adding a cutoff function as in \cite[proof of Theorem 9.11]{GT} yields standard interior estimates),
we deduce that $\forall p$ large, $v^h\in W^{2,p}_{loc}$ inside a slightly smaller domain in $\bar x(B_1)$ with local estimates independent of $h$.  Since for difference quotient $v^h$ we have $D^2v^h\in L^p_{loc}$, it follows that $D^2\bar u\in W^{1,p}_{loc}$ for all large $p$, hence by the Sobolev Embedding Theorem $D^2\bar u\in C^{\alpha}(B_{1/2})$ for all $\alpha\in(0,1)$.
\end{proof}

Now from Schauder theory, the rotated potential $\bar u(\bar x)$ is as regular as the equation allows.

\begin{corollary}
\label{cor:higher}
Let $u$ be a convex viscosity solution of \eqref{slag} on $B_1(0)$, where $\psi(x)$ is Lipschitz.  Then rotated potential $\bar u(\bar x)$ is a $C^{2,\alpha}$ solution of \eqref{rs} on $\bar x(B_1(0))$ for all $\alpha\in(0,1)$.  If $\psi\in C^{2,\alpha}$ for some $\alpha\in (0,1)$, then $\bar u\in C^{4,\alpha}$.
\end{corollary}

Indeed, since $\bar u\in C^{1,1}$, we have $D\bar u\in C^{0,1}$ and $\bar\psi\in C^{0,1}$, and Proposition \ref{prop:C2a} applies, so $D\bar u\in C^{1,\alpha}$.  If we also know $\psi\in C^{2,\alpha}$, then this means $\bar\psi\in C^{1,\alpha}$. Recalling equation \eqref{slag:lin} for the difference quotient of $\bar u$, it follows from Schauder estimates that $D\bar u\in C^{2,\alpha}$, so $\bar\psi\in C^{2,\alpha}$.  Taking two difference quotients in a similar way, we deduce $D^2\bar u\in C^{2,\alpha}$.

\section{Proof of Theorem \ref{thm}}
\label{sec:reg}

The final step is to show that $\bar\lambda_{max}<1$ on $\bar x(B_1)$.  Indeed, once we prove this, we can compare with quadratics to convert to original variables.  Let $\bar Q(\bar x)$ touch $\bar u$ from above near $\bar x$ and suppose, after lowering, that $D^2\bar Q<I_n$.  By the order preservation of rotations, we see that $Q(x)$ touches $u$ from above near $x$.  Moreover, $D^2Q<\infty$ by the transformation law 
$$
D^2Q(x)=\frac{I+A}{I-A}\,\qquad A=D^2\bar Q(\bar x).
$$
Thus $\lambda_{max}<\infty$ on $B_1$, i.e. $u\in C^{1,1}(B_1)$.  This means $\bar x(x)=cx+sDu(x)$ is Lipschitz, so the above formula with $Q$ replaced by $u$ implies $u\in C^{2,1}(B_1)$. This means $\bar x\in C^{1,1}$, and as before, we deduce $u\in C^{3,1}(B_1)$.  Since $D^2\bar u\in C^{2,\alpha}(B_1)$ and $\bar x\in C^{2,1}$, we conclude that $u\in C^{4,\alpha}(B_1)$.  
We now establish the desired inequality. 
\begin{remark}
We need $\psi \in C^{2,\alpha}$ only for the following proposition. 
\end{remark}

\begin{proposition}
\label{prop:max}
Let $u$ be a convex viscosity solution of \eqref{slag} on $B_1(0)$, with $\psi(x)\in C^{2,\alpha}$.  Then $\bar\lambda_{max}<1$ on $\bar x(B_1(0))$.
\end{proposition}

\begin{proof}

Suppose that $\bar\lambda_{max}=1$ at a point $\bar x_0$, or more generally that $1=\bar\lambda:=\bar\lambda_{1}=\cdots=\bar\lambda_m>\bar\lambda_{m+1}\ge \cdots\ge \bar\lambda_n$ at $\bar x_0$.  
It follows that the following function is $C^{2,\alpha}$ for $\bar x$ near $\bar{x_0}$: $\bar b_m:=\frac{1}{m}\sum_1^m\ln \sqrt{1+\bar\lambda_k^2}$.  Fixing $\bar x$, 
we assume $D^2\bar u$ is diagonal at $\bar x$, with $\bar u_{\bar i\bar i}(\bar x)=\bar\lambda_i(\bar x)$. Note that $\bar{u}_{\bar{ii}}$ refers to the double partial derivatives of the function $\bar{u}$ w.r.t to $\bar {x_i}$.

From \cite[Lemma 4.1]{AB} we get the following:
\begin{equation}
m\Delta_{\bar g}\bar b_m=\tilde{Z}+\sum_{i=1}^m\frac{\bar\lambda_i}{1+\bar\lambda_i^2}\bar\psi_{\bar i\bar i}-m\sum_{i=1}^m\bar \lambda_i\bar g^{\bar{i}\bar i}\bar\psi_{\bar i}\pd_{\bar i}\bar b_m \label{main_one}
\end{equation}

where 
\begin{align*}
    \tilde{Z}=\sum_{k\leq m}(1+\bar\lambda^2)\bar h^2_{\overline{kkk}}+(\sum_{i<k\leq m}+\sum_{k<i\leq m})(3+3\bar\lambda^2)\bar h_{\overline{iik}}^2+\sum_{k\leq m<i}\frac{2\bar\lambda(1+\bar\lambda\bar\lambda_i)}{\bar\lambda-\bar\lambda_i}\bar h^2_{\overline{iik}}\nonumber\\
     +\sum_{i\leq m<k}\frac{3\bar\lambda-\bar\lambda_k+\bar\lambda^2(\bar\lambda+\bar\lambda_k)}{\bar\lambda-\bar\lambda_k} {\bar h_{\overline{iik}}}^2
     +2\Bigg[\sum_{i<j<k\leq m}(3+3\bar \lambda^2)\bar h_{\overline{ijk}}^2
     +\sum_{i<j\leq m<k}[1+\frac{2\bar\lambda}{\bar\lambda-\bar\lambda_k}+\nonumber\\
     \frac{\bar\lambda^2(\bar\lambda+\bar\lambda_k)}{\bar\lambda-\bar\lambda_k}]\bar h_{\overline{ijk}}^2
     +\sum_{i\leq m<j<k}\bar\lambda[\bar\lambda_j+\bar\lambda_k+\frac{1+\bar \lambda_j^2}{\bar\lambda-\bar\lambda_j}+\frac{1+\bar\lambda_k^2}{\bar\lambda-\bar\lambda_k}]\bar h_{\overline{ijk}}^2\Bigg]
\end{align*}
with \begin{align*}\bar h_{\overline{ijk}}=\sqrt{\bar g^{\overline{ii}}}\sqrt{\bar g^{jj}}\sqrt{\bar g^{kk}}\bar u_{\overline{ijk}} \text{ and }
\bar g^{ii}=\frac{1}{1+\bar \lambda_i^2}.
\end{align*}

Again following the notation introduced in \cite[pg 11]{AB}, for each fixed $k$ in the above expression, we set 
$\bar t_{i} = \bar h_{\bar{iik}}$ and write
\begin{align*}
    \bar H^k=\bar {t}_1+...+\bar{t}_{n-1}+\bar{t}_n=\bar{t'}+\bar{t}_{n}
     \end{align*}
     where $\bar H^k$ denotes the $k$th component of the mean curvature vector (given by (\ref{mean}): $\bar H^k=\bar g^{kk}\bar\psi_{\bar k}$), i.e. the component of the mean curvature vector along $J(\bar{e_k},\bar D\bar{u}_{\bar{e_k}})$ with $\bar{e_k}$ being the $k^{th}$ eigendirection of $ D^2\bar u$.
One can re-write \[\tilde{Z}=\text{(constant phase terms)}+\tilde{Z_0}\]
where for each fixed $k\leq m$, the $k$th term of $\tilde{Z_0}$ is given by \begin{equation*}
    [\tilde{Z_0}]_k=\frac{2\bar \lambda_n}{\bar \lambda-\bar \lambda_n}[(\bar H^k)^2-2\bar H^k\bar t']. 
\end{equation*}
By constant phase terms, we denote terms without dependence on the variable phase $\psi(x)$, which are therefore the same as in the $\psi=const.$ case considered by \cite{CSY}.
The constant phase terms are nonnegative if the Hessian bounds \eqref{bounds} are true, as in \cite[Section 3]{CSY}.

 
Using \eqref{barpsi}, we note the following
\begin{align}
[\tilde{Z_0}]_k
     =\frac{\bar\lambda_n}{\bar\lambda-\bar\lambda_n}(\bar{g}^{\overline{kk}})^2(\partial_k\psi)^2(1-\bar\lambda_k)^2-\frac{2\sqrt2\bar\lambda_n}{\bar\lambda-\bar\lambda_n}\bar{g}^{\overline{kk}}(\partial_k\psi)(1-\bar\lambda_k)\bar{t}'\nonumber\\
     \geq -C(|\psi|_{C^1})(1-\bar\lambda_k)[(1-\bar\lambda_k)+1]. \label{Zterm}
\end{align}
Before dealing with (\ref{Zterm}), we would first like to deal with the second and third terms of (\ref{main_one}), since (\ref{Zterm}) can be treated similarly.

The third term of (\ref{main_one}) involving $\nabla \bar b_m$ is harmless, so it suffices to lower bound the second term, which will have two contributions.  Recalling \eqref{barpsi},
\begin{equation}
    \bar\psi_{\bar{ii}}=\frac{\pd^2}{\pd \bar x_i^2}\psi\left(\frac{\bar x-D\bar u(\bar x)}{\sqrt 2}\right)=-\frac{1}{\sqrt 2}\psi_a\partial_{\bar a}\bar\lambda_i+\frac{1}{2}\psi_{ii}(1-\bar\lambda_i)^2. \label{main_two}
\end{equation}
The first term of (\ref{main_two}) yields a harmless contribution to the maximum principle:
$$
\sum_{a=1}^n\psi_a\sum_{i=1}^m\frac{\bar\lambda_i}{1+\bar\lambda_i^2}\partial_{\bar a}\bar\lambda_i=m\sum_{a=1}^n\psi_a\partial_{\bar a}\bar b_m.
$$
For the second term of (\ref{main_two}), we start with
\begin{equation}
  \frac{1}{2}\sum_{i=1}^m\frac{\bar\lambda_i}{1+\bar\lambda_i^2}\psi_{ii}(1-\bar\lambda_i)^2\ge -c(\bar x)(1-\bar\lambda_m)^2, \label{second}
\end{equation}
for some locally bounded $c(\bar x)$.  
Next, we note that $\bar\lambda_m>0$ for $\bar x$ nearby $x_0$.  By the convexity of $\bar b(t):=\ln\sqrt{1+t^2}$ in $[0,1]$,
$$
0<\frac{\bar b(1)-\bar b(0)}{1-0}\le \frac{\bar b(1)-\bar b(t)}{1-t},
$$
so putting $t=\bar\lambda_m(\bar x)$ 
yields
\begin{equation}
(1-\bar \lambda_m)^2\le C(\bar b(1)-\bar b(\bar\lambda_m))^2\le C(m\bar b(1)-\sum_{i=1}^{m} \bar b(\bar\lambda_i))^2. \label{psi_trick}
\end{equation}

We thus conclude that in a sufficiently small neighborhood 
of $\bar x$,
$$
\Delta_{\bar g}\bar b_m\ge \langle a(\bar x),\nabla_{\bar g}\bar b_m\rangle_{\bar g}-c(\bar x)(\bar b_{m,max}-\bar b_m)
$$
for some 
bounded continuous $a,c$, where $\bar b_{m,max}=\bar b_m|_{\bar\lambda_1=\cdots=\bar\lambda_m=1}$. We make a quick note here: The bound given by (\ref{Zterm}) can be treated in an exactly similar manner as above.

The strong maximum principle still holds because the right side vanishes at an interior maximum, according to \cite[Lemma 3.4 and Theorem 3.5]{GT}.  It follows that $\bar b_{m}\equiv \bar b_{m,max}$ and $\bar\lambda_{max}\equiv 1$ on an open set containing $\bar x$.  Since $\bar x(B_1)$ is connected, we conclude this is true everywhere: 
$\bar\lambda_{max}\equiv 1$.  However, because $u(x)$ is bounded, we can touch it from above somewhere in $B_1(0)$ by a sufficiently tall quadratic $Q$.  The rotation $\bar Q$ then touches $\bar u$ from above somewhere in $\bar x(B_1)$.  But $D^2Q<\infty$ corresponds (see \cite[end of section 3]{CSY}) to $D^2\bar Q<I_n$, a contradiction.
\end{proof}

\section{Proof of Theorem \ref{thm:est}}
\label{sec:comp}
\begin{proof}
We now prove Hessian estimate \eqref{hess_est} by compactness.  If the estimate fails, then there is a sequence of $C^{4,\alpha}$ convex solutions $u_k$ and $C^{2,\alpha}$ phases $\psi_k$ to \eqref{slag} with
\begin{align*}
    \lambda_{max}[D^2u_k(0)]\to \infty,\qquad u_k(0)=Du_k(0)=0\\
    \|u_k\|_{C^1(B_{1/2}(0))}+\|\psi_k\|_{C^{2,\alpha}(B_{1/2}(0))}\le C.
\end{align*}

Since $\|u_k\|_{C^1}$ is bounded, we can pass to a subsequence and assume $u_k$ converges uniformly to $u$ on $B_{1/2}(0)$. Note that $u$ is necessarily convex, so $\tilde u:=\frac{1}{\sqrt 2}(u+\frac{1}{2}|x|^2)$, as in Section 2, is uniformly convex. This means $\partial\tilde u(B_{r}(0))$ is open and contains $B_{r/\sqrt 2}(0)$, for $0<r\le 1/2$. Moreover, if we shrink slightly and suppose $\bar x\in \partial\tilde u(B_{1/2-\delta}(0))$, then $\bar x\in\partial\tilde u_k(B_{1/2}(0))$ for large enough $k$, since if $\bar x\in \partial\tilde u_k(x_k)$, then uniform convexity yields, via \eqref{xsmall},
$$
|x-x_k|^2\le C\|\tilde u-\tilde u_k\|_{L^\infty(B_{1/2}(0)}\to 0.
$$
So the rotated sequence $\bar u_{k}$ is defined on arbitrarily large subsets of $\partial\tilde u(B_{1/2}(0))$, converging uniformly to $\bar u$ thereabouts by the order preservation of rotation.  It follows that $\bar u$ is the locally uniform limit on $\partial\tilde u(B_{1/2}(0))$ of smooth rotations $\bar u_k$.  

\medskip
The smooth rotations $\bar u_k$ have eigenvalues which blowup: $\bar\lambda_{max,k}(0)\to 1$.  To see this for $\bar u$, we use the $C^{2,\alpha}$ estimates for $\bar u_k$, noting that $\psi_k$ converges to some $\psi\in C^{2,\alpha}(B_{1/2}(0))$ in the norm of $C^{2,\alpha/2}(B_{1/2}(0))$ after taking a subsequence.  Along a subsequence, it follows that $\bar u_k$ eventually converges locally in $C^{4,\alpha/2}$ to $\bar u$ in $\partial \tilde u(B_{1/2}(0))$, so $\bar\lambda(0)=1$.  Moreover, $\bar u$ is a $C^{4,\alpha}$ solution of the rotated equation $\bar F(D^2u)=\bar\psi$, so $\bar\lambda_{max}\equiv 1$ by the strong maximum principle arguments in Proposition \ref{prop:max}.  This is a contradiction, since $\bar\lambda_{max}<1$ somewhere on $\partial\tilde u(B_{1/2}(0))$ for bounded convex $u$.
\end{proof}
\begin{remark}
In fact, a stronger Hessian estimate than \eqref{hess_est} holds:
\begin{equation}
    |D^2u(0)|\leq C_1\exp [C_2\, (osc_{B_{1/2}}u) ^{2n-2}] 
\end{equation}
where $C_1$ and $C_2$ are positive constants depending on $||\psi||_{C^{1,1}(B_{1/2})}$ and $n$.  This result follows from the methods in \cite{AB}; the proof in \cite{AB} goes through if the supercriticality condition $|\psi|\ge (n-2)\pi/2+\delta$ is replaced by the convexity condition $D^2u\ge 0$. 
A weaker estimate for such smooth solutions was obtained earlier in \cite[Theorem 8]{WTh}.  
\end{remark}

\bibliographystyle{amsalpha}
\bibliography{BiB}

\end{document}